\documentclass[12pt]{article}

\usepackage{amsmath}
\usepackage{amssymb}
\usepackage{amsthm}
\usepackage{hyperref}
\usepackage{enumerate}

\newtheorem{dfn}{Definition} [section]
\newtheorem{theorem}[dfn]{Theorem}
\newtheorem{lemma}[dfn]{Lemma}

\newtheorem{corollary}[dfn]{Corollary}
\newtheorem{conjecture}[dfn]{Conjecture}
\newcommand\abs[1]{\lvert #1\rvert}
\def\X {{\mathcal X}}
\def\T {{\mathcal T}}
\def\Se {{\mathcal S}}

\newcommand\email[1]{\footnote{\texttt{#1}}}
\newcommand\td{tree-decom\-po\-si\-tion}
\newcommand\pd{path-decom\-po\-si\-tion}
\begin{document}

\title{Partitioning $H$-minor free graphs into three subgraphs with no large components}
\author{Chun-Hung Liu%
\email{chliu@math.princeton.edu}
\thanks{Partially supported by National Science Foundation under Grant No.\ DMS-1664593.}
 \\
\small Department of Mathematics, \\
\small Princeton University,\\
\small Princeton, NJ 08544, USA
\and 
Sang-il Oum%
\email{sangil@kaist.edu} %
\thanks{Supported by the National Research Foundation of Korea (NRF) grant funded by the Korea government (MSIT) (No. NRF-2017R1A2B4005020).}\\
\small Department of Mathematical Sciences,\\\small  KAIST, \\\small Daejeon, 34141 South Korea}
\date{August 28, 2017}
\maketitle

\begin{abstract}
We prove that for every graph $H$, if a graph $G$ has no (odd) $H$ minor, then its vertex set $V(G)$ can be partitioned into three sets $X_1$, $X_2$, $X_3$ such that for each~$i$, the subgraph induced on $X_i$ has no component of size larger than a function of~$H$ and the maximum degree of~$G$. This improves a previous result of Alon, Ding, Oporowski and Vertigan~(2003) stating that $V(G)$ can be partitioned into four such sets if $G$ has no $H$ minor.
Our theorem generalizes a result of Esperet and Joret~(2014), who proved it for graphs embeddable on a fixed surface and asked whether it is true for graphs with no $H$ minor.

As a corollary, we prove that for every positive integer $t$, if a graph $G$ has no $K_{t+1}$ minor, then its vertex set $V(G)$ can be partitioned into $3t$ sets $X_1,\ldots,X_{3t}$ such that for each~$i$, the subgraph induced on $X_i$ has no component of size larger than a function of~$t$. This corollary improves a result of Wood~(2010), which states that $V(G)$ can be partitioned into $\lceil 3.5t+2\rceil$ such sets.

\end{abstract}

\section{Introduction}
The famous Four Color Theorem states that every planar graph $G$ admits a partition of its vertex set into four sets $X_1,X_2,X_3,X_4$ such that for $1\le i\le 4$, every component of the subgraph~$G[X_i]$ induced on $X_i$ has at most one vertex. Certainly there are planar graphs whose vertex set cannot be partitioned into three such sets. 
However, Esperet and Joret~\cite{ej} proved that the number of sets can be reduced to three, if we relax each $X_i$ to induce a subgraph having no component of size larger than a function of the maximum degree of~$G$.

\begin{theorem}[{Esperet and Joret~\cite{ej}}] \label{3color surface}
Let $\Sigma$ be a surface of Euler genus $g$. 
If a graph $G$ is embeddable on $\Sigma$ and has maximum degree at most $\Delta$, then $V(G)$ can be partitioned into three sets $X_1,X_2,X_3$ such that for $1\le i\le 3$, every component of~$G[X_i]$ has at most $(5\Delta)^{2^g-1}(15\Delta)^{(32\Delta+8)2^g}$ vertices.
\end{theorem}

The number of sets in Theorem~\ref{3color surface} is best possible, since a $k \times k$ triangular grid has maximum degree six but its vertex set cannot be partitioned into two sets such that each set induces a subgraph with no component of size less than $k$ by the famous HEX lemma~\cite{Gale1979}.
In contrast, Alon, Ding, Oporowski, and Vertigan~\cite{adov} showed that for graphs of bounded tree-width and bounded maximum degree, it is possible to partition the vertex set into two sets inducing subgraphs having no large components.

\begin{theorem}[{Alon et al.~\cite[Theorem 2.2]{adov}}%
\footnote{In \cite{adov}, Theorem~\ref{3color bdd tw} is stated without requiring $w\ge 3$. However, \cite{adov} cites \cite[(3.7)]{DO1995}, which requires $w\ge 3$.
However, Theorem~\ref{3color bdd tw} is true even if $w< 3$, because a stronger statement was proved by Wood~\cite{Wood2009}.}] \label{3color bdd tw}
Let $w\ge 3$ and $\Delta$ be positive integers.
If a graph $G$ has maximum degree at most $\Delta$ and tree-width at most $w$, then $V(G)$ can be partitioned into $X_1$, $X_2$ such that for $1\le i\le2$, every component of~$G[X_i]$ has at most $24w\Delta$ vertices.
\end{theorem}

It was pointed out by Esperet and Joret {[private communication, 2015]} that the condition of maximum degree mentioned in Theorem \ref{3color bdd tw} cannot be removed.
See Theorem \ref{no bdd diameter} for details.

Though it is impossible to partition all planar graphs of bounded maximum degree into two induced subgraphs with components of bounded size, it is possible to partition them such that the tree-width of every component is small.
More precisely, DeVos, Ding, Oporowski, Sanders, Reed, Seymour, and Vertigan~\cite{ddosrsv} proved the following result, which was conjectured by Thomas~\cite{t}.
A graph $H$ is a \emph{minor} of a graph $G$ if a graph isomorphic to $H$ can be obtained from a subgraph of  $G$ by contracting edges. 

\begin{theorem}[{DeVos et al.~\cite{ddosrsv}}] \label{2color into bdd tw}
  For every graph $H$, there exists an integer~$N$ such that if $H$ is not a minor of~$G$, then $V(G)$ can be partitioned into two sets $X_1$, $X_2$ such that the tree-width of~$G[X_i]$ is at most $N$ for $1\le i\le 2$.
\end{theorem}

Alon, Ding, Oporowski, and Vertigan~\cite{adov} combined Theorems~\ref{3color bdd tw} and~\ref{2color into bdd tw} to prove the following theorem.

\begin{theorem}[{Alon et al.~\cite[Theorem 6.7]{adov}}] \label{4color minor-closed}
 For every graph $H$ and every positive integer $\Delta$, there exists an integer $N$ such that if $H$ is not a minor of a graph $G$ of the maximum degree at most $\Delta$, then $V(G)$ can be partitioned into four sets $X_1,X_2,X_3,X_4$ such that for $1\le i\le4$, every component of~$G[X_i]$ has at most $N$ vertices.
\end{theorem}

In this paper, we prove the following strengthening of Theorems~\ref{3color surface} and \ref{4color minor-closed} and answer a question of Esperet and Joret~\cite[Question 5.1]{ej}.

\begin{theorem} \label{3color minor free}
  For every graph $H$ and every positive integer $\Delta$, there exists an integer $N$ such that if $H$ is not a minor of a graph $G$ of the maximum degree at most $\Delta$, then $V(G)$ can be partitioned into three sets $X_1,X_2,X_3$ such that for $1\le i\le 3$, every component of~$G[X_i]$ has at most $N$ vertices.
\end{theorem}

\paragraph{Strengthening to odd minors}
Indeed, we prove a stronger theorem in terms of odd minors as follows.
A graph $H$ is an \emph{odd minor} of a graph $G$ if there exists a set $\{(T_v)\}_{v\in V(H)}$ of vertex-disjoint subgraphs of $G$ that are trees such that 
each tree $T_v$ is properly colored by colors $1$ and $2$
and for each edge $uw$ of $H$, there exists an edge joining $T_u$ and $T_w$ whose ends have the same color.

\begin{theorem} \label{3color odd minor free}
  For every graph $H$ and every positive integer $\Delta$, there exists an integer $N$ such that if $H$ is not an \emph{odd} minor of a graph $G$ of the maximum degree at most $\Delta$, then $V(G)$ can be partitioned into three sets $X_1,X_2,X_3$ such that for $1\le i\le 3$, every component of~$G[X_i]$ has at most $N$ vertices.
\end{theorem}
Since every odd minor of a graph $G$ is also a minor of $G$, Theorem~\ref{3color odd minor free} trivially implies Theorem~\ref{3color minor free}.

Interestingly Demaine, Hajiaghayi, and Kawarabayashi~\cite{DHK2010} proved 
a result analogous to Theorem~\ref{2color into bdd tw} for odd minors, claiming that graphs with no odd $H$-minor can be partitioned into two induced subgraphs of bounded tree-width. This with Theorem~\ref{3color bdd tw} would imply that graphs with no odd $H$-minor having bounded maximum degree can be partitioned into $4$ induced subgraphs each having no large components. Theorem~\ref{3color odd minor free} reduces the number of induced subgraphs from four to three.

\paragraph{Applications to a weaker version of Hadwiger's conjecture}
As an application of Theorem~\ref{3color minor free}, we investigate the following relaxation of Hadwiger's conjecture: what is the minimum $k$ as a function of~$t$ such that  
for some $N$, every graph $G$ with no $K_{t+1}$ minor admits a partition of $V(G)$ into $k$ sets $X_1,X_2,\ldots,X_k$ with the property that each $G[X_i]$ has no component on more than $N$ vertices?
Hadwiger's conjecture~\cite{hadwiger}, if true, would imply that $k=t$ (with $N=1$).
Kawarabayashi and Mohar~\cite{km} proved that $k \leq \lceil 15.5(t+1) \rceil$, and Wood~\cite{w} proved that $k \leq \lceil 3.5t+2 \rceil$.
We improve these results by using a recent result of Edwards, Kang, Kim, Oum, and Seymour~\cite{ekkos}.

\begin{theorem}[\cite{ekkos}] \label{partitioning into bounded max deg}
For every positive integer $t$, there exists $s$ such that if $K_{t+1}$ is not a minor of a graph $G$, then $V(G)$ can be partitioned into $t$ sets $X_1,X_2,\ldots,X_{t}$ such that for $1\le i\le t$, $G[X_i]$ has maximum degree at most~$s$.
\end{theorem}

\begin{theorem}
For every positive integer $t$, there exists $N$ such that if $K_{t+1}$ is not a minor of a graph $G$, then $V(G)$ can be partitioned into $3t$ sets $X_1,X_2,\ldots,X_{3t}$ such that for $1\le i\le 3t$, every component of~$G[X_i]$ has at most $N$ vertices.
\end{theorem}

\begin{proof}
By Theorem \ref{partitioning into bounded max deg}, there exists an integer $s$ such that $V(G)$ can be partitioned into $t$ sets $V_1,V_2,\ldots,V_t$ such that the maximum degree of~$G[V_i]$ is at most $s$ for $1\le i\le t$.
By Theorem~\ref{3color minor free}, there exists an integer $N$ depending only on $t$ such that for  $1\le i\le t$, $V_i$ can be partitioned into three sets $V_{i1}, V_{i2}, V_{i3}$ and each of~$G[V_{i1}]$, $G[V_{i2}]$, $G[V_{i3}]$ has no component having size larger than $N$ vertices.
\end{proof}

\bigskip

In this paper, \emph{graphs} are simple.
A \emph{$k$-coloring} of a graph $G$ is a function mapping the vertices of~$G$ into the set $\{1,2,\ldots,k\}$.
A \emph{monochromatic component} is a component of the subgraph induced by the vertices of the same color in a given $k$-coloring.
The \emph{size} of a component is the number of its vertices.
For a graph $G$ and a set $X$ of vertices,
let $N_G(X)$ be the set of vertices not in $X$ but adjacent to some vertex in $X$
and $N_G[X] = N_G(X) \cup X$. For a vertex $v$ of a graph $G$, let $N_G(v)=N_G(\{v\})$.

The proof of Theorem~\ref{3color odd minor free} uses the machinery in the Graph Minors series of Robertson and Seymour
and the structure theorem of graphs with no odd minors by Geelen, Gerards, 
Reed, Seymour, and Vetta~\cite{GGRSV2009}.
A theorem by Robertson and Seymour~\cite{rsXVI} states that every graph that excludes a fixed graph as a minor can be ``decomposed'' into pieces satisfying certain structure properties.
We will review some tools in the Graph Minors series and modify the aforementioned pieces such that they are relatively easier to be $3$-colored with small monochromatic components in Section~\ref{sec:minors}.
In Section~\ref{sec:monochromatic components}, we complete the proof of Theorem~\ref{3color odd minor free} by first $3$-coloring the aforementioned pieces and then extending the coloring to the whole graph.
Finally, we will make some remarks in Section~\ref{sec:remarks}.

\section{Structure theorems}
\label{sec:minors}
In this section, we review some notions in the Graph Minors series of Robertson and Seymour and derive a structure for graphs without a fixed graph as a minor.

A \emph{tree-decomposition} of a graph $G$ is a pair $(T,\X)$ such that $T$ is a tree and $\X = \{X_t: t \in V(T)\}$ is a collection of subsets of~$V(G)$ with the following properties.
\begin{itemize}
	\item $\bigcup_{t \in V(T)} X_t = V(G)$.
	\item For every $e \in E(G)$, there exists $t \in V(T)$ such that $X_t$ contains both ends of~$e$.
	\item For every $v \in V(G)$, the subgraph of~$T$ induced by $\{t \in V(T): v \in X_t\}$ is connected.
\end{itemize}
For every $t \in V(T)$, $X_t$ is called the \emph{bag} of~$t$.
The \emph{width} of~$(T,\X)$ is $\max\{\abs{ X_t } : t \in V(T)\}-1$.
The \emph{adhesion} of~$(T,\X)$ is $\max\{\abs{X_t \cap X_{t'}}: tt' \in E(T)\}$.
A \td{} $(T,\X)$ is a \emph{\pd} if $T$ is a path.
The \emph{tree-width} of~$G$ is the minimum width of a \td{} of~$G$.

A \emph{separation} of a graph $G$ is an ordered pair $(A,B)$ of subgraphs with $A \cup B=G$ and $E(A \cap B) = \emptyset$, and the \emph{order} of a separation $(A,B)$ is $\abs{ V(A) \cap V(B) }$.
A~\emph{tangle} $\T$ in $G$ of \emph{order} $\theta$ is a set of separations of~$G$, each of order less than $\theta$ such that
\begin{enumerate}
	\item[(T1)] for every separation $(A,B)$ of~$G$ of order less than $\theta$, either $(A,B) \in \T$ or $(B,A) \in \T$;
	\item[(T2)] if $(A_1, B_1), (A_2,B_2), (A_3,B_3) \in \T$, then $A_1 \cup A_2 \cup A_3 \neq G$;
	\item[(T3)] if $(A,B) \in \T$, then $V(A) \neq V(G)$.
\end{enumerate}
Tangles were first introduced by Roberson and Seymour~\cite{rsX}.
We call (T1), (T2) and (T3) the first, second and third tangle axiom, respectively.
For a subset $Z$ of $V(G)$ with $\abs{Z}<\theta$, we define $\T-Z$ to be the set of all separations $(A',B')$ of $G-Z$ of order less than $\theta-\abs{Z}$ such that there exists $(A,B) \in \T$ with $Z \subseteq V(A) \cap V(B)$, $A'=A-Z$ and $B'=B-Z$.
We remark that $\T-Z$ is a tangle in $G-Z$ of order $\theta-\abs{Z}$ by~\cite[Theorem 8.5]{rsX}.

Given a graph $H$, an \emph{$H$-minor} of a graph $G$ is a map $\alpha$ with domain $V(H) \cup E(H)$ such that the following hold.
\begin{itemize}
	\item For every $h \in V(H)$, $\alpha(h)$ is a nonempty connected subgraph of~$G$. 
	\item If $h_1$ and $h_2$ are different vertices of~$H$, then $\alpha(h_1)$ and $\alpha(h_2)$ are disjoint.
	\item For each edge $e$ of~$H$ with ends $h_1,h_2$, $\alpha(e)$ is an edge of~$G$ with one end in $\alpha(h_1)$ and one end in $\alpha(h_2)$; furthermore, if $h_1=h_2$, then $\alpha(e) \in E(G)-E(\alpha(h_1))$.
        \item If $e_1, e_2$ are two different edges of~$H$, then $\alpha(e_1) \neq \alpha(e_2)$.
\end{itemize}
We say that \emph{$G$ contains an $H$-minor} if such a function $\alpha$ exists.
A tangle~$\T$ in $G$ \emph{controls} an $H$-minor $\alpha$ if 
$\T$ has no $(A,B)$ of order less than $\abs{V(H)}$ such that $V(\alpha(h)) \subseteq V(A)$ for some $h\in V(H)$.

A \emph{society} is a pair $(S,\Omega)$, where $S$ is a graph and $\Omega$ is a cyclic permutation of a subset $\bar{\Omega}$ of~$V(S)$.
For a nonnegative integer $\rho$, a society $(S,\Omega)$ is a \emph{$\rho$-vortex} if for all distinct $u,v \in \bar{\Omega}$, there do not exist $\rho+1$ mutually disjoint paths of~$S$ between $I \cup \{u\}$ and $J \cup \{v\}$, where $I$ is the set of vertices in $\bar{\Omega}$ after $u$ and before $v$ in the order $\Omega$, and $J$ is the set of vertices in $\bar{\Omega}$ after $v$ and before $u$.
For a society $(S,\Omega)$ with $\bar{\Omega} = \{v_1,v_2,\ldots,v_{\abs{\bar{\Omega}}}\}$ in order, a \emph{vortical decomposition} of~$(S,\Omega)$ is a path-decomposition $(t_1t_2\cdots t_{\lvert \bar{\Omega} \rvert}, \X )$ such that the $i$-th bag of~$\X$ contains the $i$-th vertex $v_i$ for each $i$.

\begin{theorem}[{Robertson and Seymour~\cite[(8.1)]{rsIX}}] \label{path decomp of vortex}
  Every $\rho$-vortex  has a vortical decomposition with adhesion at most $\rho$.
\end{theorem}

A \emph{segregation} of a graph $G$ is a set $\Se$ of societies such that 
\begin{itemize}
	\item $S$ is a subgraph of~$G$ for every $(S, \Omega) \in \Se$, and $\bigcup \{S: (S,\Omega) \in \Se\}=G$,
	\item for every distinct $(S,\Omega)$ and $(S', \Omega') \in \Se$, $V(S \cap S') \subseteq \bar{\Omega} \cap \overline{\Omega'}$ and $E(S \cap S') = \emptyset$.
\end{itemize}
We write $V(\Se) = \bigcup \{\bar{\Omega}: (S, \Omega) \in \Se\}$.
For a tangle $\T$ in $G$, a segregation $\Se$ of~$G$ is \emph{$\T$-central} if for every $(S,\Omega) \in \Se$, there is no $(A,B) \in \T$ with $B \subseteq S$.

A \emph{surface} is a nonnull compact connected $2$-manifold without boundary.
Let $\Sigma$ be a surface.
For every subset $\Delta$ of~$\Sigma$, we denote the closure of~$\Delta$ by $\bar{\Delta}$ and the boundary of~$\Delta$ by $\partial\Delta$.
An \emph{arrangement} of a segregation $\Se = \{(S_1, \Omega_1), \ldots, (S_k, \Omega_k)\}$ in $\Sigma$ is a function $\alpha$ with domain $\Se \cup V(\Se)$, such that the following hold.
\begin{itemize}
	\item For $1 \leq i \leq k$, $\alpha(S_i, \Omega_i)$ is a closed disk $\Delta_i \subseteq \Sigma$, and $\alpha(x) \in \partial\Delta_i$ for each $x \in \overline{\Omega_i}$.
	\item For $1 \leq i \leq k$, if $x \in \Delta_i \cap \Delta_j$, then $x=\alpha(v)$ for some $v \in \overline{\Omega_i} \cap \overline{\Omega_j}$.
	\item For all distinct $x,y \in V(\Se)$, $\alpha(x) \neq \alpha(y)$.
	\item For $1 \leq i \leq k$, $\Omega_i$ is mapped by $\alpha$ to the natural order of~$\alpha(\overline{\Omega_i})$ determined by $\partial\Delta_i$.
\end{itemize}
An arrangement is \emph{proper} if $\Delta_i \cap \Delta_j = \emptyset$ whenever 
$\abs{ \overline{\Omega_i} }, \abs{ \overline{\Omega_j}} >3$,
for all $1 \leq i < j \leq k$.

An \emph{O-arc} is a subset homeomorphic to a circle, and a \emph{line} is a subset homeomorphic to $[0,1]$.
A \emph{drawing} $\Gamma$ in a surface $\Sigma$ is a pair $(U,V)$, where $V \subseteq U \subseteq \Sigma$, $U$ is closed, $V$ is finite, $U-V$ has only finitely many arc-wise connected components, called \emph{edges}, and for every edge $e$, either $\bar{e}$ is a line with set of ends $\bar{e} \cap V$, or $\bar{e}$ is an O-arc and $\lvert \bar{e} \cap V \rvert =1$.
The components of~$\Sigma-U$ are called \emph{regions}.
The members of~$V$ are called \emph{vertices}.
If $v$ is a vertex of a drawing $\Gamma$ and $e$ is an edge or a region of~$\Gamma$, we say that $e$ is \emph{incident with} $v$ if $v$ is contained in the closure of~$e$.
Note that the incidence relation between vertices and edges of $\Gamma$ defines a multigraph, and we say that $\Gamma$ is a \emph{drawing} of a multigraph $G$ in $\Sigma$ if $G$ is defined by this incident relation.
In this case, we say that $G$ is \emph{embeddable} in $\Sigma$, or $G$ can be \emph{drawn} in $\Sigma$.
A drawing is \emph{$2$-cell} if $\Sigma$ is connected and every region is an open disk.

A drawing $\Gamma=(U,V)$ in $\Sigma$ is the \emph{skeleton} of a proper arrangement $\alpha$ of a segregation $\Se$ in $\Sigma$ if $V=\bigcup_{v \in V(\Se)} \alpha(v)$ and  $U$ consists of the boundary of~$\alpha(S,\Omega)$ for each $(S,\Omega) \in \Se$ with $\lvert \bar{\Omega} \rvert = 3$, and a line in $\alpha(S',\Omega')$ with ends $\overline{\Omega'}$ for each $(S',\Omega') \in \Se$ with $\lvert \overline{\Omega'} \rvert=2$.
Note that we do not add any edges into the skeleton for $(S,\Omega)$ with $\lvert \bar{\Omega} \rvert \leq 1$ or $\lvert \bar{\Omega} \rvert >3$.

A segregation $\Se$ of~$G$ is \emph{maximal} if there exists no segregation $\Se'$ such that $\{(S,\Omega) \in \Se: \lvert \bar{\Omega} \rvert > 3\} = \{(S',\Omega') \in \Se': \lvert \overline{\Omega'} \rvert > 3\}$ and for every $(S,\Omega) \in \Se$ with $\lvert \bar{\Omega} \rvert \leq 3$, there exists $(S',\Omega') \in \Se'$ with $\lvert \overline{\Omega'} \rvert \leq 3$ such that $S' \subseteq S$, and the containment is strict for at least one society.
Note that if $\Se$ is maximal, then for every $(S,\Omega) \in \Se$ with $\abs{\bar\Omega}\le 3$ and every $v \in \bar{\Omega}$, there exist $\lvert \bar{\Omega} \rvert-1$ paths in $S$ from $v$ to $\bar{\Omega}-\{v\}$ intersecting only in $v$.
In particular, the maximum degree of the skeleton of a proper arrangement of a maximal segregation of~$G$ is at most the maximum degree of~$G$.

By taking advantage of a theorem by Robertson and Seymour~{\cite[Theorem~(9.2)]{rsXIV}}, the following statement is an easy corollary of a theorem in Dvo\v{r}\'{a}k~{\cite[Theorem 7]{d}} by choosing the function $\phi$ in {\cite[Theorem 7]{d}} to be the constant function $4d+5$.
(We omit the statements of {\cite[Theorem~(9.2)]{rsXIV}} and {\cite[Theorem 7]{d}} as they require a couple of definitions to be formally stated but will not be further used in the rest of the paper.)

\begin{corollary} \label{stronger excluding minor}
For every graph $L$, there exists an integer $\kappa$ such that for every positive integer $d$, there exist integers $\theta, \xi, \rho$ with the following property.
If a graph $G$ has a tangle $\T$ of order at least $\theta$ controlling no $L$-minor of~$G$, then there exist $Z \subseteq V(G)$ with $\abs{Z}\le \xi$, a maximal $(\T-Z)$-central segregation $\Se=\Se_1 \cup \Se_2$ of~$G-Z$ with $\abs{\Se_2}\le \kappa$ and a proper arrangement $\alpha$ of~$\Se$ in some surface $\Sigma$ in which $L$ cannot be drawn, such that $\abs{\bar{\Omega}}\le 3$ for all $(S,\Omega) \in \Se_1$,  every member in $\Se_2$ is a $\rho$-vortex, and the skeleton $G'$ of $\alpha$ of $\Se$ in $\Sigma$ satisfies the following.
	\begin{enumerate}
        \item $G'$ is $2$-cell embedded in $\Sigma$. 
		\item For every $(S,\Omega) \in \Se_2$, there exists a closed disk $D_S$ in $\Sigma$ containing $\alpha(S)$ and disjoint from $\bigcup_{(S',\Omega') \in \Se_2-\{(S,\Omega)\}} D_{S'}$ such that $D_S$ contains every vertex of~$G'$ that can be connected by a path in $G'$ of length at most $d$ from a vertex in $\bar{\Omega}$.
		\item For distinct $(S,\Omega),(S',\Omega') \in \Se_2$, there exists no path of length at most $2d+2$ in $G'$ from $\overline{\Omega}$ to $\overline{\Omega'}$.
	\end{enumerate}
\end{corollary}

Let $G_0$ be a drawing in a surface $\Sigma$ with $k$ pairwise disjoint closed disks $D_1$,
$D_2$, $\ldots$, $D_k$ such that each disk intersects $G_0$ only in vertices of~$G_0$ and contains no vertex of~$G_0$ in its interior.
For $1 \leq i \leq k$, let $v_{i,1},v_{i,2},\ldots,v_{i,n_i}$ be the vertices of~$G_0 \cap \partial D_i$ appearing on $\partial D_i$ in order.
For a positive integer $w$, a graph $G$ is an \emph{outgrowth by $k$
  $w$-rings of a graph $G_0$ in $\Sigma$}~\cite{rsXVI} if 
\begin{itemize}
\item there exist $k$ societies $(S_1,\Omega_1)$, $(S_2,\Omega_2)$, $\ldots$, $(S_k,\Omega_k)$  such that 
$G=G_0\cup \bigcup_{i=1}^k S_i$ and 
$\Omega_i=\{v_{i,1},v_{i,2},\ldots,v_{i,n_i}\}$ in order and $S_i \cap G_0 = \overline{\Omega_i}$ for $1\le i\le k$, 
\item for $1 \leq i \leq k$,  $S_i$ has a path-decomposition $(t_{i,1}t_{i,2}\cdots t_{i,n_i},\X_i)$  of width at most $w$ such that $v_{i,j} \in X_{i,j}$ for $1 \leq j \leq n_i$, where $X_{i,j}$ is the bag at $t_{i,j}$.
\end{itemize}
In addition, for $d\ge 0$, we say that $G$ is \emph{$d$-local} if $G_0$ satisfies the following.
	\begin{itemize}
		\item For $1 \leq i \leq k$, there exists a closed disk $D$ in $\Sigma$ containing $D_i$ and disjoint from $\bigcup_{j \neq i} D_j$ such that $D$ contains every vertex of~$G_0$ that can be connected by a path in $G_0$ of length at most $d$ from a vertex in $V(G_0) \cap \partial D_i$.
		\item For $1 \leq i < j \leq k$, $G_0$ has no path of length at most $2d+2$ from a vertex in $V(G_0) \cap \partial D_i$ to a vertex in $V(G_0) \cap \partial D_j$.
	\end{itemize}

Let $\Se$ be a segregation of a graph $G$.
Assume that for every $(S,\Omega) \in \Se$ with $\lvert \bar{\Omega} \rvert > 3$, there exists a path-decomposition $(P_S=t_1t_2\cdots t_{\lvert \bar{\Omega} \rvert},\X_S)$ such that the bag at $t_i$, denoted by $X_{S,i}$, contains the $i$-th vertex $v_{S,i}$ in $\bar{\Omega}$, for every $1 \leq i \leq \abs{ \bar{\Omega} }$.
For every $(S,\Omega) \in \Se$ with $\lvert \bar{\Omega} \rvert >3$, let $G_S$ be the graph obtained from the subgraph of~$S$ induced by $\bar{\Omega} \cup \bigcup_{i=1}^{\lvert \bar{\Omega} \rvert-1} (X_{S,i} \cap X_{S,i+1})$ by adding three new vertices $x_{S,i,1}$, $x_{S,i,2}$, $x_{S,i,3}$ for each $i\in \{1,2,\ldots,\abs{\bar{\Omega} }\}$ with the same set of neighbors 
\[
N_G(X_{S,i} - (\{v_{S,i}\} \cup X_{S,i-1} \cup X_{S,i+1})) \cap 
(\{v_{S,i}\} \cup X_{S,i-1}  \cup  X_{S,i+1}),
\]
where  $X_{S,0}=X_{S,\lvert \bar{\Omega} \rvert+1} = \emptyset$.

Let $G_0$ be the skeleton of a proper arrangement $\alpha$ of~$\Se$ in a surface $\Sigma$.
We define the \emph{extended skeleton} of~$\alpha$ of~$\Se$ in $\Sigma$ \emph{with respect to $\{(P_S,\X_S): (S,\Omega) \in \Se, \lvert \bar{\Omega} \rvert >3\}$} to be the graph obtained from the disjoint union of~$G_0$ and $G_S$ for every $(S,\Omega) \in \Se$ with $\lvert \bar{\Omega} \rvert >3$ by identifying the copies of the $i$-th vertex of~$\bar{\Omega}$ in $G_0$ and $G_S$ for each $(S,\Omega) \in \Se$ with $\lvert \bar{\Omega} \rvert >3$ and for every $1 \leq i \leq \lvert \bar{\Omega} \rvert$.
Note that if there are at most $\kappa$ members $(S,\Omega)$ of~$\Se$ with $\lvert \bar{\Omega} \rvert >3$ and the adhesion of each $(P_S,\X_S)$ is at most $\rho$, then the extended skeleton of~$\alpha$ of~$\Se$ is an outgrowth by $\kappa$ $(2\rho+3)$-rings of~$G_0$ in $\Sigma$.
Furthermore, if $\Se$ is a maximal segregation, then the maximum degree of the extended skeleton of~$\alpha$ of~$\Se$ in $\Sigma$ is at most $\max\{3\Delta, 2\rho+1\}$, where $\Delta$ is the maximum degree of~$G$.

\begin{theorem} \label{structure thm for excluding minor locally planar}
For every graph $L$ and positive integer $d$, there exist integers $\kappa, \theta, \xi, \rho$ with the following property.

If a graph $G$ has a tangle $\T$  of order at least $\theta$ controlling no $L$-minor of~$G$, then there exist $Z \subseteq V(G)$ with $\abs{Z}\le \xi$ and a maximal $(\T-Z)$-central segregation $\Se$ of~$G-Z$ 
with a proper arrangement $\alpha$ in a surface $\Sigma$ in which $L$ cannot be drawn,
such that if $\Se_1=\{(S,\Omega)\in \Se : \abs{\bar \Omega}\le 3\}$ and $\Se_2=\Se-\Se_1$, then 
	\begin{enumerate}
		\item $\lvert \Se_2 \rvert \leq \kappa$ and every $(S,\Omega)\in \Se_2$ is a $\rho$-vortex with a vortical decomposition $(P_S,\X_S)$ of adhesion at most $\rho$, 
		\item the extended skeleton of~$\alpha$ of~$\Se$ in $\Sigma$ with respect to $\{(P_S,\X_S): (S,\Omega) \in \Se_2\}$ is a $d$-local outgrowth by $\kappa$ $(2\rho+3)$-rings of the skeleton of~$\alpha$ of~$\Se$ in $\Sigma$, whose maximum degree is at most $\max\{3\Delta, 2\rho+1\}$, where $\Delta$ is the maximum degree of~$G$.
	\end{enumerate}
\end{theorem}

\begin{proof}
Let $\kappa,\theta,\xi,\rho$ be the numbers, $\Se=\Se_1 \cup \Se_2$ the segregation of~$G$, $\Sigma$ the surface, $\alpha$ the arrangement of~$\Se$ in $\Sigma$ obtained by applying Corollary~\ref{stronger excluding minor}.
By Theorem~\ref{path decomp of vortex}, for every $(S,\Omega) \in \Se_2$, $S$ has a vortical decomposition $(P_S,\X_S)$ of adhesion at most $\rho$.
Therefore, the extended skeleton of~$\alpha$ of~$\Se$ in $\Sigma$ with respect to $\{(P_S,\X_S): (S,\Omega) \in \Se_2\}$ is a $d$-local outgrowth by $\kappa$ $(2\rho+3)$-rings of the skeleton of~$\alpha$ of~$\Se$ in $\Sigma$.
Since $\Se$ is maximal, the maximum degree of the extended skeleton of~$\alpha$ of~$\Se$ is at most $\max\{3\Delta, 2\rho+1\}$
\end{proof}

\section{Monochromatic components}
\label{sec:monochromatic components}

For an integer $q>0$, 
a \emph{$q$-necklace with chain $v_1v_2 \ldots v_n$} is a multigraph $G$ with $V(G) = \{v_1,v_2,\ldots,v_n\}$ such that 
\begin{itemize}
	\item $v_1v_2\cdots v_nv_1$ is a cycle $C$, 
	\item $G$ contains pairwise edge-disjoint complete subgraphs $M_1,M_2,\ldots,M_k$ each having at most $q$ vertices such that $E(G) - E(C)= \bigcup_{i=1}^k E(M_i)$, and
	\item there exist no integers $i,j,a,b,c,d$ with $i\neq j$ and $a<b<c<d$ such that $\{v_a,v_c\} \subseteq V(M_i)$ and $\{v_b,v_d\} \subseteq V(M_j)$.
\end{itemize}
Note that every $2$-connected outerplanar multigraph is a $2$-necklace.

\begin{lemma} \label{tw necklace}
Every $q$-necklace has tree-width at most $\max\{q-1,2\}$.
\end{lemma}
\begin{proof}
  Let $G$ be a $q$-necklace with $n$ vertices $v_1,v_2,\ldots,v_n$ and $k$ complete subgraphs $M_1,M_2,\ldots,M_k$ each having at most $q$ vertices.
  Since outerplanar multigraphs have tree-width at most $2$, we may assume that $q\ge 3$ and $k\ge 1$.

  We claim that there is a tree-decomposition $(T,\X)$ of width at most $q-1$. We proceed by induction on $k$. 
  If $k=1$, then it is trivial to find such a tree-decomposition $(T,\X)$, as the graph $G$ is isomorphic to a graph obtained from $M_1$ by adding many paths. In $(T,\X)$, one bag is $M_1$ and other bags have at most three vertices.

  Now suppose that $k>1$. We may assume that $V(M_1)=\{v_{i_1},v_{i_2},\ldots,v_{i_q}\}$ with $1\le i_1<i_2<\cdots<i_q\le n$. We may assume that $i_1=1$ by rotating labels.
Let $i_{q+1}=n+1$ and $v_{n+1}=v_1$.
  For $j\in \{1,2,\ldots,q\}$, let $W_j=\{v_{i_j},v_{i_j+1},\ldots,v_{i_{j+1}}\}$.
  It is easy to see that $G[W_j]$ is a $q$-necklace and so it has a tree-decomposition $(T_j,\X_i)$ of width at most $q-1$. Since $v_{i_j}$ is adjacent to $v_{i_{j+1}}$,
  $T_j$ has a node $t_j$ whose bag contains $v_{i_j}$ and $v_{i_{j+1}}$.
  
  Let $T$ be the tree obtained from the disjoint union of all $T_j$ by adding a node $t$ adjacent to all $t_j$.  Let $M_1$ be the bag corresponding to $t$ and we assign bags to all other nodes of~$T$ according to $\X_j$ for some $j$. It is easy to see that this is a tree-decomposition of width at most $q-1$.
\end{proof}

We use the same idea of the proof of \cite[Lemma 8.1]{dt} to prove the following generalization.

\begin{lemma} \label{tw extended necklace}
For an integer $q\ge 3$, let $H$ be a $q$-necklace with chain $u_1u_2\cdots u_n$.
Let $(S,\Omega)$ be a society 
with a vortical decomposition $(t_1t_2\cdots t_n, \X)$ of width $w$.
If $G$ is the multigraph obtained from the disjoint union of~$S$ and $H$ by identifying $u_i$ with the $i$-th vertex of $\Omega$ for each $1 \leq i \leq n$, then $G$ has tree-width at most $q(w+1)-1$.
\end{lemma}

\begin{proof}
Let $X_i$ be the $i$-th bag of $\X$.
By Lemma~\ref{tw necklace}, $H$ has a \td{} $(T,\X')$ of width at most $q-1$.
We denote $\X'$ by $\{X_t': t\in V(T)\}$.
For every $t \in V(T)$, define $X''_t = X'_t \cup \bigcup\{X_i: u_i \in X'_t,1 \leq i \leq n\}$ and $\X''=\{X''_t:t\in V(T)\}$.
Since there exists a path in $H$ passing through $u_1$, $u_2$, $\ldots$, $u_n$ in order, $(T,\X'')$ is a tree-decomposition of~$G$ and \[\lvert X''_t \rvert \leq (w+1)\lvert X'_t \rvert \leq q(w+1).\]
So the width of~$(T,\X'')$ is at most $q(w+1)-1$.
\end{proof}

For a positive integer $k$ and a graph $G$, 
we say that a $k$-coloring $c$ of a subgraph $H$ of $G$ can be \emph{extended} to a $k$-coloring of~$G$ or can be \emph{extended} to $G$ if $G$ has a $k$-coloring $c'$ such that $c'(v) = c(v)$ for every $v \in V(H)$.

\begin{lemma} \label{mono comp local}
Let $\Delta,k,w$ and $g$ be positive integers.
Let $\Sigma$ be a surface of Euler genus $g$, and let $d=(5\Delta)^{2^g-1}(15\Delta)^{(32\Delta+8)2^g}$.
If $G$ is a $d$-local outgrowth by $k$ $w$-rings of a graph $G_0$ in $\Sigma$ and $G$ has maximum degree $\Delta$, then $G$ can be $3$-colored in such a way that every monochromatic component has at most 
$48 d^4 w\Delta^5$ vertices.
\end{lemma}

\begin{proof}
We may assume that $\Delta\ge 3$.
Let $D_1,D_2,\ldots,D_k$ be the closed disks and $(S_1,\Omega_1),\ldots,(S_k,\Omega_k)$ the societies mentioned in the definition of an outgrowth $G$ by $k$ $w$-rings of~$G_0$ in $\Sigma$.

By Theorem~\ref{3color surface}, there exists a $3$-coloring $c$ of $G-\bigcup_{i=1}^k V(S_i)$ such that every monochromatic component has at most $d$ vertices.
Let $L_i$ be the set of vertices of $G-\bigcup_{k=1}^k V(S_i)$ that has a monochromatic path to a vertex in $N_G(V(S_i))$ with respect to $c$. Since $G$ is $d$-local, $L_i\cap L_j=\emptyset$ for $1\le i<j\le k$. Let $G_i=G[L_i\cup V(S_i)]$.
To prove this lemma, it suffices to show that for every $1 \leq i \leq k$, the $3$-coloring on $G_i-V(S_i)$ can be extended to a $3$-coloring of $G_i$ such that every monochromatic component of $G_i$ has at most $48 d^4 w\Delta^5$ vertices.

For $1 \leq i \leq k$, define $H_i$ to be the multigraph obtained from $G[V(S_i)]$ by first adding a cycle passing through $\overline{\Omega_i}$ in order and
adding a complete graph on $N_G(V(C))\cap V(S_i)$
for each monochromatic component $C$ of~$G_i-V(S_i)$.
Since each $C$ contains at most $d$ vertices, each of added complete subgraphs has at most $d\Delta$ vertices.
Since $G$ is $d$-local, $H_i[\overline{\Omega_i}]$ is a $d\Delta$-necklace.
Hence, the tree-width of~$H_i$ is at most $d\Delta(w+1)-1 $ by Lemma~\ref{tw extended necklace}, and the maximum degree of~$H_i$ is at most $(d\Delta-1)\Delta+2 \leq d\Delta^2$, as $\Delta\ge 3$.
By Theorem~\ref{3color bdd tw}, there exists a $3$-coloring of $H_i$ (in fact, a  $2$-coloring) such that every monochromatic component of $H_i$ contains at most $24 \cdot d\Delta(w+1) \cdot  d\Delta^2 $ vertices.

Now, we extend the $3$-coloring $c$ of $G_i-V(S_i)$ to 
the $3$-coloring $c'$ of $G_i$
by taking the $3$-coloring of $H_i$  on $V(S_i)$.
Let $Q$ be a monochromatic component of~$G_i$ with respect to $c'$.
We know that $\abs{V(Q)\cap V(S_i)}\le 24 \cdot d\Delta(w+1) \cdot  d\Delta^2$. Since $H_i$ has maximum degree at most $d\Delta^2$,
each vertex of $H_i$ may join at most $d\Delta^2$ distinct monochromatic components of $H_i-V(S_i)$, each having at most $d$ vertices.
Thus, $Q$ contains at most 
\[
(24 \cdot d\Delta(w+1) \cdot  d\Delta^2)\cdot d\Delta^2 \cdot d
\le 48 d^4 w\Delta^5
\]
vertices.
This completes the proof.
\end{proof}

The following simple lemma is a stronger statement of~\cite[Observation 3.9]{ej}.
This lemma is obvious, so we omit the proof.

\begin{lemma} \label{recolor vertex}
Let $G$ be a graph of maximum degree $\Delta$ and $Z$ a subset of~$V(G)$.
Assume that $G$ has a coloring such that every monochromatic component has size at most an integer $k$.
If we recolor some vertices in $Z$, then the union of the monochromatic components intersecting in $Z$ in the new coloring has at most $\lvert Z \rvert(\Delta k+1)$ vertices, and every monochromatic component disjoint from $Z$ in the new coloring has at most $k$ vertices.
\end{lemma}

We use the following theorem of Geelen et al.~\cite{GGRSV2009} on odd minors.
\begin{theorem}[Geelen et al.~{\cite[Theorem 13]{GGRSV2009}}]\label{thm:oddminor}
  There is a constant $c$ such that if $G$ contains a $K_t$-minor $\alpha$
  where $t=\lceil c \ell \sqrt{\log 12\ell}\rceil$, 
  then either $G$ contains an odd $K_\ell$-minor,
  or there exists a set $X$ of vertices with $\abs{X}<8\ell$ such that the (unique) block $U$ of $G-X$ that intersects all branch vertices of $\alpha$ disjoint from $X$ is bipartite.
\end{theorem}

Our main theorem, Theorem~\ref{3color odd minor free}, is an immediate corollary of the following stronger theorem by taking $Y=\emptyset$.

\begin{theorem}
For every graph $W$ and positive integer $\Delta$, 
there exists an integer $\eta$ such that 
if $W$ is not an odd minor of a graph $G$ of maximum degree at most $\Delta$, 
then for every subset $Y$ of $V(G)$ with $\abs{Y}\le \eta$, 
every $3$-coloring of $Y$ can be extended to that of $G$ satisfying the following.

	\begin{enumerate}[(i)]
		\item The union of all monochromatic components of~$G$ meeting $Y$ contains at most $\abs{Y}^2\Delta$ vertices.
		\item Every monochromatic component of~$G$ contains at most $\eta^2\Delta$ vertices.
	\end{enumerate}
\end{theorem}

\begin{proof}
We may assume that $\Delta>1$.
Since $G$ does not have an odd $W$-minor, 
by Theorem~\ref{thm:oddminor}, 
there exist sufficiently large integers $c$ and $t=\lceil c \lvert V(W) \rvert \sqrt{\log{12 \lvert V(W) \rvert}} \rceil$ such that 
if $G$ has a $K_t$-minor $\alpha$, then 
it has a set $X$ of vertices such that $\abs{X}<8\abs{V(W)}$
and 
the (unique) block $U$ of $G-X$ intersecting all branch vertices of $\alpha$ is bipartite.
Let $L=K_t$.

Let $d=(5\Delta)^{2^g-1}(15\Delta)^{(32\Delta+8)2^g}$, where $g$ is the maximum genus of a surface in which $L$ cannot be drawn. 
(If $L$ is planar, then let $g=0$.)
Let $\kappa,\xi,\theta$ and $\rho$ be given by Theorem~\ref{structure thm for excluding minor locally planar} for $L$ and $d$.
We may assume that $\theta>8\abs{V(W)}+1$.
Let $M=48 d^4 (2\rho+3) (3\Delta+2\rho)^5$ and $\eta=2000 \rho \theta^3 M \Delta^6$.

We proceed by induction on $\abs{V(G)}$.
It is trivial if $\abs{V(G)}\le 1$, because $\eta\ge 1$ and $\Delta>0$. Thus we may assume that $\abs{V(G)}\ge 2$.
Let $Y$ be a subset of $V(G)$ with at most $\eta$ vertices.
We may assume that $Y$ is nonempty, because otherwise we can add one vertex to $Y$.
Let $c_Y:Y\to\{1,2,3\}$ be a given $3$-coloring of~$Y$. 
We say that a $3$-coloring $c$ of $G$ is \emph{$Y$-good} with respect to $c_Y$ if it extends $c_Y$ and satisfies the conditions 1 and 2 of the theorem.
Suppose that $G$ has no $Y$-good $3$-coloring with respect to $c_Y$.
Note that for the condition 2, it is unnecessary to consider monochromatic components meeting $Y$ because it follows from the condition 1.

For a subset $X$ of $V(G)$, we write $1_X$ to denote a $3$-coloring of $X$ coloring all vertices of $X$  by $1$. Similarly we define $2_X$ and $3_X$.

\medskip\noindent {\bf Claim 1:} $\lvert Y \rvert > \frac{\eta}{\Delta^2}$.

\smallskip
\noindent {\bf Proof of Claim 1:}
Let $Y_1=N_G(Y)$ and $Y_2=N_G(Y\cup Y_1)$.
Note that $\abs {Y_2 } \leq \abs{Y} \Delta(\Delta-1) \leq \eta$.
By permuting colors, we may assume that $3\in c_Y(Y)$.
We apply the induction hypothesis to $G-(Y \cup Y_1)$ with 
the $3$-coloring $1_{Y_2}$ of $Y_2$
to obtain a $Y_2$-good 3-coloring $c$ of~$G-(Y \cup Y_1)$.
Let $c'$ be a $3$-coloring of $G$ such that 
\[
c'(v)=
\begin{cases}
  c_Y(v) &\text{if }v\in Y,\\
  2 & \text{if }v\in Y_1,\\
  c(v) &\text{if }v\notin Y\cup Y_1.
\end{cases}
\]
No monochromatic component of $G$ with respect to $c'$ can meet both $Y\cup Y_1$ and $Y_2$ because $Y_1$ and $Y_2$ are colored differently.
Since $c$ is $Y_2$-good, 
every monochromatic component of~$G$ disjoint from $Y\cup Y_1$ contains at most $\eta^2\Delta$ vertices.
The union of the monochromatic components of~$G$ meeting $Y\cup Y_1$ contains at most $\abs{ Y \cup Y_1 } \leq (\Delta+1) \abs Y $ vertices.
If $(\Delta+1)\abs Y\le \abs{Y}^2\Delta$, then it implies the condition 1 and condition 2 for monochromatic components meeting $Y\cup Y_1$.
If $\lvert Y \rvert \geq 2$, then $\lvert Y \cup Y_1 \rvert \leq \lvert Y \rvert^2 \Delta$; if $\lvert Y \rvert=1$, then the monochromatic component of~$G$ meeting $Y$ contains exactly one vertex of color $3$.
Therefore $c'$ is $Y$-good, contradicting our assumption.
$\Box$

\medskip\noindent {\bf Claim 2:} There exists no separation $(A,B)$ of~$G$ of order less than $\theta$ such that $\lvert (V(A)-V(B)) \cap Y\rvert \geq 3\theta$ and $\lvert  (V(B)-V(A)) \cap Y\rvert \geq 3\theta$.

\smallskip
\noindent {\bf Proof of Claim 2:}
Suppose that $G$ has a separation $(A,B)$ of order less than $\theta$ such that 
$a=\lvert (V(A)-V(B)) \cap Y \rvert \geq 3\theta$ and $b=\lvert (V(B)-V(A)) \cap Y \rvert \geq 3\theta$.
Let 
$Y_A = (Y\cup V(B))\cap V(A)$ and 
$Y_B = (Y \cup V(A))\cap V(B)$.
Then, 
\begin{align*}
  \abs{Y}^2-(\abs{Y_A}^2+\abs{Y_B}^2)
  &\ge (a+b)^2 -(a+\theta)^2-(b+\theta)^2\\
  &= \frac{2ab}{3}+\frac{2ab}{3}+\frac{2ab}{3}-2(a+b)\theta -2\theta^2\\
  &\ge 2a\theta+2b\theta+6\theta^2-2(a+b)\theta-2\theta^2=4\theta^2>0.
\end{align*}

Now we shall construct a desired $3$-coloring of~$G$.
We first color vertices in $Y_A-Y(=Y_B-Y)$ arbitrary.
Since $\abs {V(A)} \le\abs{ V(G)} - 3\theta$ and $\abs{Y_A} \leq \abs{Y}-2\theta$, we can apply the induction hypothesis to the graph $A$ with $Y_A$ precolored.
Similarly, we can further apply the induction hypothesis to the graph $B$ with $Y_B$ precolored.
So by merging the 3-colorings of~$A$ and $B$, we obtain a $3$-coloring of~$G$.
Let $U$ be the union of the monochromatic components of~$G$ either meeting $Y$, or meeting both $A$ and $B$.
Note that every component of~$U$ meets  $Y_A \cup Y_B$.
By the induction hypothesis, $U$ contains at most $(\lvert Y_A \rvert^2 + \lvert Y_B \rvert^2)\Delta \leq \lvert Y\rvert^2 \Delta$ vertices.
On the other hand, the induction hypothesis implies that every monochromatic component of~$G$ disjoint from $Y$ contains at most $\eta^2 \Delta$ vertices.
Therefore, $G$ has a $Y$-good 3-coloring, contradicting our assumption.
$\Box$

\medskip
We define $\T$ to be the set of all separations $(A,B)$ of~$G$ of order less than $\theta$ such that $\lvert (V(B)-V(A)) \cap Y \rvert \geq 3 \theta$.

\medskip\noindent {\bf Claim 3:} $\T$ is a tangle in $G$ of order $\theta$.

\smallskip
\noindent {\bf Proof of Claim 3:}
Observe that there exists no separation $(A,B)$ of order less than $\theta$ such that $\lvert (V(A)-V(B)) \cap Y \rvert < 3\theta$ and $\lvert (V(B)-V(A)) \cap Y \rvert < 3\theta$, since otherwise $\lvert Y \rvert < 7\theta \leq \frac{\eta}{\Delta^2}$, contradicting Claim 1.
So $\T$ satisfies the first tangle axiom.

Suppose that there exist $(A_j,B_j) \in \T$ for $1 \leq j \leq 3$ such that $A_1 \cup A_2 \cup A_3 = G$.
By Claim 2, $\lvert (V(A_j)-V(B_j)) \cap Y \rvert < 3 \theta$ for $1 \leq j \leq 3$.
So $\lvert V(A_j) \cap Y \rvert < 4\theta$ for $1 \leq j \leq 3$.
As a result, $\lvert Y \rvert \leq \sum_{j=1}^3 \lvert Y \cap V(A_j) \rvert < 12 \theta \leq \frac{\eta}{\Delta^2}$, a contradiction.
Hence the second tangle axiom holds.

If $V(A)=V(G)$ for some $(A,B) \in \T$, then $\lvert Y \rvert < 4\theta \leq \frac{\eta}{\Delta^2}$ by Claim 2, a contradiction.
Therefore, $\T$ is a tangle of order $\theta$.
$\Box$

\smallskip
\noindent{\bf Claim 4:} $\T$ controls no $L$-minor.

\smallskip
\noindent {\bf Proof of Claim 4:}
Suppose that $\T$ controls an $L$-minor $\alpha$. 
Since $\alpha$ is an $L$-minor in $G$, by Theorem~\ref{thm:oddminor}, 
there exists a set $X$ of vertices such that $\abs{X}\le 8\abs{V(W)}$ such that the unique block $U$ of $G-X$ intersecting all branch vertices of $\alpha$ 
disjoint from $X$ is bipartite.

Let $C_1,C_2,\ldots,C_m$ be the list of induced subgraphs of $G-X-V(U)$ such that
$G-X-V(U)$ is the disjoint union of $C_1,C_2,\ldots,C_m$,
each component of $C_i$ has the same set of neighbors in $U$,
and for $i\neq j$, 
the set of neighbors of $C_i$ in $U$
is not equal to that of $C_j$ in $U$.

As $U$ is a block of $G-X$, 
each $C_i$ has at most one neighbor in $U$.
For each $i\in \{1,2,\ldots,m\}$, 
Let $A_i$ be the 
subgraph of $G$ induced by the union of $X$, $V(C_i)$, and the set of all neighbors of $V(C_i)$ in $U$.
Let $B_i$ be the subgraph of $G-E(A_i)$ induced on $V(G)-V(C_i)$.
Note that 
$(A_i,B_i)$ is a separation of $G$ such that 
$\abs{V(A_i)\cap V( B_i)}\le \abs{X}+1\le 8\abs{V(W)}+1<\theta$.
As $\T$ is a tangle of $G$ of order $\theta$, 
$(A_i,B_i)\in\T$ or $(B_i,A_i)\in\T$. Since $\T$ controls an $L$-minor $\alpha$
and all branch vertices of $\alpha$ disjoint from $X$ intersect $U$, 
$(A_i,B_i)\in\T$.
By the definition of $\T$, we deduce that 
\[
\abs{(V(B_i)-V(A_i))\cap Y}\ge 3\theta.
\]
By Claim 2, $\abs{(V(A_i)-V(B_i))\cap Y}<3\theta$.

First we properly color $U$ by colors $1$ and $2$
and color all vertices in $X$ by color $3$.
This coloring 
of $G[V(U)\cup X]$ has the property that 
each monochromatic component has at most $\abs{X}<\theta $ vertices.
Then we recolor vertices in $Y\cap (V(U)\cup X)$ by its given color. 
By Lemma~\ref{recolor vertex}, this new coloring of $G[V(U)\cup X]$
has the property that the union of all monochromatic components 
intersecting $Y\cap (V(U)\cup X)$ has at most $\eta(\Delta \theta+1)\le 2\eta \Delta \theta$ vertices.

For each $i\in\{1,2,\ldots,m\}$, 
let $Y_i'=(Y\cap V(A_i))\cup (V(A_i)\cap V(B_i))$.
Note that $\abs{Y_i'}<4\theta$.
By the induction hypothesis, 
there exists a $Y_i'$-good coloring $f_i$ of $A_i$
extending the coloring of $G[V(U)\cup X]$ given in the previous step
such that 
the union of all monochromatic components of $A_i$ in $f_i$ intersecting $Y_i'$
has at most $(4\theta)^2 \Delta$ vertices
and every monochromatic component of $A_i$ in $f_i$ has at most 
$\eta^2\Delta$ vertices.

Let $f$ be a $3$-coloring of $G$
obtained by combining the coloring of $G[V(U)\cup X]$ and the coloring $f_i$ for 
each $i\in\{1,2,\ldots,m\}$.
This coloring $f$ is well-defined
and furthermore 
the union of all monochromatic components in $G$
intersecting $Y\cup X\cup \bigcup_{i=1}^m (V(A_i)\cap V(U))$ has 
at most $(2\eta \Delta \theta) (4\theta)^2 \Delta^2 \le \abs{Y}^2\Delta$
vertices.
In addition, each monochromatic component in $G$
not intersecting  $Y\cup X\cup \bigcup_{i=1}^m (V(A_i)\cap V(U))$ 
has at most $\eta^2\Delta$ vertices.
This completes the proof.
$\Box$

\medskip

Now we may assume that $\T$ controls no $L$-minor.
By Theorem~\ref{structure thm for excluding minor locally planar}, there exist $Z \subseteq V(G)$ with $\lvert Z \rvert \leq \xi$ and a maximal $(\T-Z)$-central segregation $\Se=\Se_1 \cup \Se_2$ of~$G-Z$ properly arranged by an arrangement $\alpha$ in a surface $\Sigma$ in which $L$ cannot be drawn, where every $(S,\Omega) \in \Se_1$ has the property that $\lvert \bar{\Omega} \rvert \leq 3$, and $\lvert \Se_2 \rvert \leq \kappa$ and every member $(S,\Omega)$ in $\Se_2$ is a $\rho$-vortex with a vortical decomposition $(P_S,\X_S)$ of adhesion at most $\rho$ such that the extended skeleton of $\alpha$ of~$\Se$ in $\Sigma$ with respect to $\{(P_S,\X_S): (S,\Omega) \in \Se_2\}$, denoted by $G'$, is a $d$-local outgrowth by $\kappa$ $(2\rho+3)$-rings of the skeleton of~$\alpha$ of~$\Se$ in $\Sigma$ and  the maximum degree of~$G'$ is at most $\max\{3\Delta,2\rho+1\} \leq 3\Delta+2\rho$.

Let $c'$ be a $3$-coloring of~$G'$ given by Lemma~\ref{mono comp local} such that every monochromatic component of~$G'$ with respect to $c'$ contains at most $M$ vertices.
Let $G''$ be the graph obtained from the disjoint union of~$G'$ and $G[Z]$ by adding the edges of~$G$ between $Z$ and $V(G') \cap V(G)$.
Then we extend $c'$ to a 3-coloring $c''$ of~$G''$ by coloring every vertex in $Z$ by color 1.
Then each monochromatic component in $G''$ contains at most $\max(\abs Z,1) (M \Delta +1)$ vertices.
Since $\abs Z  < \theta$, we know that $\max(\abs Z,1)  (M\Delta  +1) \leq 2\theta M\Delta$.
Note that the maximum degree of~$G''$ is still at most $3\Delta+2\rho$.

For each $(S,\Omega) \in \Se_1$, let $Q_S = G[V(S) \cup Z]$
and \[Y_S=\bar\Omega\cup Z \cup (Y\cap V(S))).\]
Since $\abs{\bar{\Omega} \cup Z}\le \abs{Z}+3<\theta$,
$(Q_S,G-(V(S)-\bar\Omega)-E(Q_S))$ is a separation of $G$ having order less than $\theta$.
Since $\Se$ is $(\T-Z)$-central, $(Q_S,G-(V(S)-\bar\Omega)-E(Q_S))\in \T$ and therefore $\abs{Y\cap V(S)}<4\theta$. So 
$\abs{Y_S} \le \abs{Z}+3+\abs{Y\cap V(S)}
< \xi+3+4\theta\le 6\theta\Delta\le \eta$.

For each $(S,\Omega) \in \Se_2$  and $1 \leq i \leq \abs{\bar{\Omega}}$, 
let $X_{S,i}$ be the $i$-th bag of~$\X_S$, which contains the $i$-th vertex $v_{S,i}$ in $\bar{\Omega}$;
let $Q_{S,i}=G[X_{S,i} \cup Z]$, 
$  B_{S,i} = Z \cup (X_{S,i}\cap (X_{S,i-1} \cup  X_{S,i+1} \cup \{v_{S,i}\}) )$ 
where 
$X_{S,0}=X_{S,\abs{\bar{\Omega}}+1}=\emptyset$;
let \[Y_{S,i}=B_{S,i}\cup (N_G(B_{S,i})\cap X_{S,i}) \cup (Y\cap X_{S,i}).\]
Note that there exists $(A,B) \in \T$ with $V(A) \cap V(B) = B_{S,i}$ and $G[X_{S,i}] \subseteq A$, since $\abs{B_{S,i}}\le 2\rho +1+\xi<\theta$ and $\Se$ is $(\T-Z)$-central.
Thus, $\abs{Y\cap X_{S,i}}<4\theta$ and therefore $\abs{Y_{S,i}}
\le \abs{B_{S,i}\cup N_G(B_{S,i})\cup (Y \cap X_{S,i})}
< \theta(\Delta+1)+4\theta \leq 6\theta\Delta \leq \eta$.

For $(S,\Omega)\in \Se_2$ and $1\le i\le \abs{\bar\Omega}$,  
let $x_{S,i,1},x_{S,i,2},x_{S,i,3}$ be the vertices of $G'$ mentioned in the definition of the extended skeleton, and let $W_{S,\Omega}$ be a minimum set with $W_{S,\Omega} \subseteq \{x_{S,i,1},x_{S,i,2},x_{S,i,3}\}$ and $\lvert W_{S,\Omega} \rvert  \leq \min\{\abs {Y \cap X_{S,i} \cap N_G(B_{S,i}) },3\}$.

Now we define a new $3$-coloring $c'''$ of~$G''$ by the following rule.
\begin{itemize}
	\item $c'''(v)=c_Y(v)$ if $v \in Y \cap V(G'')$.
	\item For $(S,\Omega) \in \Se_2$ and $1\le i\le \abs{\bar{\Omega}}$, 
	  define $c'''$ on $\{x_{S,i,1},x_{S,i,2},x_{S,i,3}\}$ such that $c'''(\{x_{S,i,1},x_{S,i,2},x_{S,i,3}\}) \supseteq c_Y(Y \cap X_{S,i} \cap N_G(B_{S,i}))$ and $c'''(v)=c''(v)$ for every $v \in \{x_{S,i,1},x_{S,i,2},x_{S,i,3}\}-W_{S,\Omega}$.
	\item $c'''(v)=c''(v)$ for other vertices of $G''$.
\end{itemize}
Let \begin{align*}
 Y' 
 & = \{v \in V(G''): c'''(v) \neq c''(v)\} \cup (Y \cap V(G'')) \cup Z \\ 
 &\cup \bigcup_{(S,\Omega) \in \Se_1, Y \cap V(S)-\bar{\Omega} \neq \emptyset}\bar{\Omega},
\end{align*}
\begin{align*}
 Y_1 
 & = \{v\in Y: v\in V(S)-\bar\Omega\text{ for some }(S,\Omega)\in \Se_1\} \\ 
 &\cup \{v \in Y: v \in X_{S,i}-(B_{S,i} \cup N_G(B_{S,i})) \text{ for some } (S,\Omega) \in \Se_2, 1 \leq i \leq \abs{\bar{\Omega}}\},
\end{align*}
and $Y_2=Y-Y_1$.
Since $X_{S,i} \cap N_G(B_{S,i})$ are pairwise disjoint for different pairs of  $(S,\Omega)\in\Se_2$ and $i$, $\abs{Y'} \leq \abs{Y_2} + \theta+3\abs{Y_1} \leq 4 \abs{Y}$.
Hence, the union of the monochromatic components in $G''$ with respect to $c'''$ intersecting $Y'$ contains at most $4\abs{Y} ((3\Delta+2\rho)(2\theta M\Delta )+1) \leq 48\rho\theta M\abs{Y} \Delta^2$ by Lemma~\ref{recolor vertex}.
And every monochromatic component in $G''$ with respect to $c'''$ disjoint from $Y'$ has at most $2\theta M\Delta$ vertices.

For $(S,\Omega)\in \Se_1$, let $c_S$ be a $3$-coloring of $Y_S$ such that \[
c_S(v)=
\begin{cases}
  c_Y(v) & \text{if }v\in Y,\\
  c'''(v) & \text{if }v\in \bar\Omega\cup Z,
\end{cases}
\]  for $v\in Y_S$.
As $\abs{Y_S}\le \eta$, we can apply the induction hypothesis to $Q_S$ with the $3$-coloring $c_S$ to obtain a $Y_S$-good $3$-coloring $c_S'$ of $Q_S$.

For $(S,\Omega)\in \Se_2$ and $1\le i\le \abs{\bar\Omega}$, let $c_{S,i}$ be a $3$-coloring of $Y_{S,i}$ such that \[
c_{S,i}(v)=
\begin{cases}
  c_Y(v) & \text{if }v\in Y,\\
  c'''(v) &\text{if }v\in \bar\Omega\cup Z\cup (X_{S,i}\cap (X_{S,i-1}\cup X_{S,i+1})),\\
  c'''(x_{S,i,1}) & \text{if }v\in (N_G(B_{S,i})\cap X_{S,i})-Y,
\end{cases}
\] for $v\in Y_{S,i}$. As $\abs{Y_{S,i}}\le \eta$, we can apply the induction hypothesis to $Q_{S,i}$ with the $3$-coloring $c_{S,i}$ to obtain a $Y_{S,i}$-good $3$-coloring $c_{S,i}'$ of $Q_{S,i}$.

Let $c$ be a $3$-coloring of $G$ such that 
\[c(v)=
\begin{cases}
  c'''(v) & \text{if } v\in V(G''), \\
  c_S'(v) & \text{if } v\in V(S)-\bar\Omega \text{ for some }(S,\Omega)\in \Se_1,\\
  c_{S,i}'(v) & \text{if } v\in X_{S,i}-B_{S,i}
  \\
  &\quad\text{ for some }(S,\Omega)\in \Se_2\text{ and }1\le i\le \abs{\bar\Omega}
\end{cases}
\]
for $v\in V(G)$.

We now claim that $c$ is a $Y$-good $3$-coloring of $G$.
We say that a subgraph $R$ of $G$ is \emph{hiding} if either there exists $(S,\Omega)\in \Se_1$ such that $V(R) \subseteq V(S)-\bar{\Omega}$, or there exists $(S,\Omega) \in \Se_2$ and $1 \leq i \leq \abs{\bar{\Omega}}$ such that $V(R) \subseteq X_{S,i}-(B_{S,i} \cup N_G(B_{S,i}))$.

Let $U$ be the union of monochromatic components of $G$ meeting $Y$.
For the condition 1, we need to show that $\abs{V(U)}\le \abs{Y}^2\Delta$.

Firstly let us count the vertices of $U$ that are in hiding components.
For each hiding monochromatic component $R$, $R$ contains a vertex in $Y_1$ and has at most $25\theta^2\Delta^3$ vertices by the properties of $c_S'$ and $c_{S,i}'$. 
Thus, $U$ has at most $25\theta^2\Delta^3\abs{Y_1}$ vertices in hiding components.

Secondly let us count vertices of $U$ in non-hiding components.
Let $U'$ be the graph obtained from $U$ by deleting $V(U) \cap V(S)-\bar\Omega$ and adding edges on $V(U) \cap \bar{\Omega}$ such that $U'[\bar{\Omega}]$ is a complete subgraph for every $(S,\Omega) \in \Se_1$, 
and identifying the vertices in 
$V(U)\cap (X_{S,i}-B_{S,i})$
of color $j$ in the $3$-coloring $c$ into a vertex $u_{S,i,j}$ for each $(S,\Omega) \in \Se_2$, $1 \leq i \leq \lvert \bar{\Omega} \rvert$ and $1 \leq j \leq 3$.
Note that $U'$ is isomorphic to a subgraph of $G''$.
Furthermore, for every $(S,\Omega) \in \Se_2$, $1 \leq i \leq \abs{\bar{\Omega}}$ and $1 \leq j \leq 3$, whenever $u_{S,i,j}$ exists, there exists $k$ with $1 \leq k \leq 3$ such that $c'''(x_{S,i,k})=j$, by the definition of $c'''$.
So we may assume that $U'$ is a subgraph of $G''$ with the coloring $c'''$.
Every component of $U'$ meets $Y'$, since every non-hiding component of $U$ either meets $(Y \cap V(G''))\cup Z$, or meets $\bar{\Omega}$ for some $(S,\Omega) \in \Se_1$ with $Y \cap V(S)-\bar{\Omega} \neq \emptyset$, or meets both $Y \cap X_{S,i}$ and $X_{S,i} \cap N_G(B_{S,i})$ for some $(S,\Omega) \in \Se_2$ and $1 \leq i \leq \abs{\bar{\Omega}}$.
Therefore, $U'$ contains at most $48\rho\theta M \abs{Y} \Delta^2$ vertices.

For each vertex $v$ in a non-hiding component of $U$ but not in $U'$, $v$ is either 
	\begin{itemize}
		\item contained in a monochromatic component of $Q_S$ meeting $Y_S \cap V(U')$ with respect to $c_S'$ for some $(S,\Omega) \in \Se_1$, or 
		\item contained in a monochromatic component of $Q_{S,i}$ meeting $X_{S,i} \cap N_G(B_{S,i})$ with respect to $c_{S,i}'$ for some $(S,\Omega) \in \Se_2$ and $1 \leq i \leq \abs{\bar{\Omega}}$ such that $\{x_{S,i,1},x_{S,i,2},x_{S,i,3}\} \cap V(U') \neq \emptyset$.
	\end{itemize}
Since $\Se$ is maximal, for every vertex $v$ of $G-Z$, there exist at most $\Delta$ societies $(S,\Omega) \in \Se_1$ such that $v \in V(S)$, so there are at most $\abs{U'} \Delta$ such societies in $\Se_1$ mentioned in the former case; since $\abs{\bigcup_{(S,\Omega)\in \Se_2} \bigcup_{1 \leq i \leq \abs{\bar{\Omega}}} \{x_{S,i,1},x_{S,i,2},x_{S,i,3}\} \cap V(U')} \leq \abs{U'}$, so there are at most $\abs{U'}$ such $Q_{S,i}$ mentioned in the latter case.
By the properties of $c_S'$, the union of all monochromatic components mentioned in the former case contains at most $(5\theta\Delta)^2 \cdot \abs{U'} \Delta$ vertices; by the properties of $c_{S,i}'$, the union of all monochromatic components mentioned in the latter case contains at most $(5\theta\Delta)^2 \cdot \abs{U'}$ vertices.
Hence, the number of vertices in some non-hiding components of $U$ but not in $U'$ contains at most $25\theta^2\abs{U'}\Delta^2(\Delta+1) \leq 1200\rho\theta^3M\abs{Y}\Delta^4(\Delta+1)$ vertices.

Consequently, $U$ contains at most $25\theta^2\Delta^3 \abs{Y_1} + 1200\rho \theta^3 M \abs{Y} \Delta^4(\Delta+1) \leq 2000\rho\theta^3M\abs{Y}\Delta^5 \leq \abs{Y}^2\Delta$ vertices, by Claim 1 and the assumption $\Delta \geq 2$.
This proves that $c$ satisfies condition 1.

Let $R$ be a monochromatic component of $G$ not meeting $Y$ with respect to $c$.
For condition 2, it suffices to show that $R$ contains at most $\eta\Delta^2$ vertices.
It is clear that $R$ contains at most $\max\{25\theta^2\Delta^3, \eta\Delta^2\} \leq \eta \Delta^2$ vertices if $R$ is hiding by the properties of $c_S$ and $c_{S,i}$.
So we may assume that $R$ is not hiding.

Construct $R'$ from $R$ as we constructed $U'$ from $U$.
That is, let $R'$ be the graph obtained from $R$ by deleting $V(R) \cap V(S)-\bar\Omega$ and adding edges on $V(R) \cap \bar{\Omega}$ such that $R'[\bar{\Omega}]$ is a complete subgraph for every $(S,\Omega) \in \Se_1$, 
and identifying the vertices in $V(R) \cap X_{S,i} \cap N_G(B_{S,i})$ of color $j$ in the $3$-coloring $c$ into a vertex $u_{S,i,j}$ for each $(S,\Omega) \in \Se_2$, $1 \leq i \leq \lvert \bar{\Omega} \rvert$ and $1 \leq j \leq 3$.
We may again assume that $R'$ is a subgraph of $G''$ with the coloring $c'''$.
Since $R$ is connected, $R'$ is connected.
Hence, $R'$ is a monochromatic component of $G''$ with respect to $c'''$ and contains at most $48\rho\theta M \abs{Y} \Delta^2$ vertices.

For each vertex $v$ in $R$ but not in $R'$, $v$ is either 
	\begin{itemize}
		\item contained in a monochromatic component of $Q_S$ meeting $Y_S \cap V(R')$ with respect to $c_S'$ for some $(S,\Omega) \in \Se_1$, or 
		\item contained in a monochromatic component of $Q_{S,i}$ meeting $X_{S,i} \cap N_G(B_{S,i})$ with respect to $c_{S,i}'$ for some $(S,\Omega) \in \Se_2$ and $1 \leq i \leq \abs{\bar{\Omega}}$ such that $\{x_{S,i,1},x_{S,i,2},x_{S,i,3}\} \cap V(R') \neq \emptyset$.
	\end{itemize}
Therefore, the same argument shows that the number of vertices of $R$ but not in $R'$ is at most $25\theta^2\abs{R'}\Delta^2(\Delta+1)$ vertices.
As a result, $R$ contains at most $\abs{R'}(1+25\theta^2\Delta^2(\Delta+1)) \leq  2000\rho\theta^3M\abs{Y_2}\Delta^5 \leq \eta \Delta^2$.
This shows that $c$ satisfies condition 2 and completes the proof.
\end{proof}

\section{Concluding remarks}
\label{sec:remarks}

We remark that Theorems~\ref{3color bdd tw} and \ref{3color minor free} are best possible in the sense that it is impossible to partition the vertex set into three sets such that each set induces a subgraph of bounded diameter.
The following observation is due to Esperet and Joret.
Recall that every graph with bounded tree-width does not contain a large grid as a minor.

\begin{theorem}[Esperet and Joret {[private communication, 2015]}] \label{no bdd diameter}
For every positive integers $w,d$, there exists a graph $G$ of tree-width at most $w$ such that for every $w$-coloring of~$G$, there exists a monochromatic component of~$G$ with diameter greater than $d$.
\end{theorem}

\begin{proof}
We shall construct graphs $G_i$ of tree-width at most $i$ for every $i \geq 1$ such that every $i$-coloring of~$G$ has a monochromatic component of diameter greater than $d$ recursively. 
Define $G_1$ to be the path on $d$ vertices.
Clearly, $G_1$ has tree-width one and every $1$-coloring of~$G_1$ contains a monochromatic component of diameter greater than $d$.

Assume that we have constructed the graph $G_{i-1}$ of tree-width at most $i-1$ such that every $(i-1)$-coloring of~$G$ has a monochromatic component of diameter greater than $d$.
Let $n=\abs{V(G_{i-1})}$.
Let $T$ be the rooted $n$-ary tree with root $r$ such that every internal node of~$T$ has degree $n$, and the distance between $r$ and any leaf of~$T$ is $d$.
For every node $t$ of~$T$, we create a copy $H_t$ of~$G_{i-1}$, and we denote the vertices of~$H_t$ by $u_{t,1},\ldots,u_{t,n}$.
For every internal node $t$ of~$T$, we denote the children of~$t$ by $c_{t,1},c_{t,2},\ldots,c_{t,n}$.
Then we construct $G_i$ from the disjoint union of $H_t$ for all nodes $t$ of $T$ by adding a new vertex $v$ adjacent to all vertices of~$H_r$ for the root  $r$  of $T$ and adding edges $u_{t,j}u'$ for every non-leaf $t$ of~$T$, $1 \leq j \leq n$ and $u' \in V(H_{c_{t,j}})$.

Now we prove that $G_i$ has the desired property.
Suppose that $f$ is a $i$-coloring of~$G_i$ such that every monochromatic component has diameter at most $d$.
As $G_{i-1}$ has the desired property, $V(H_t)$ receives exactly $i$ colors by $f$ for every vertex $t$ of~$T$.
In particular, each $H_t$ contains a vertex $x_t$ with $f(x_t)=f(v)$.
Since $T$ contains a path $rt_1t_2\cdots t_{d}$ of length $d$, $vx_rx_{t_1}x_{t_2}\cdots x_{t_{d}}$ is a monochromatic path of length $d+1$, a contradiction.

In addition, every block of~$G_i$ is obtained from a copy of~$G_{i-1}$ by adding a vertex.
So the tree-width of~$G_i$ is at most the one more than the tree-width of~$G_{i-1}$.
This completes the proof.
\end{proof}

Note that the graphs $G_2$ and $G_3$ mentioned in the proof of Theorem~\ref{no bdd diameter} are outerplanar and planar, respectively.
So Theorem~\ref{3color surface} cannot be improved in the same direction, either.
On the other hand, it is well known that every graph of tree-width at most $w$ contains a vertex of degree at most $w$ and hence can be properly colored by $w+1$ colors.
So Theorem~\ref{no bdd diameter} is the best possible.

Esperet and Joret [private communication, 2015] also point out that the construction of $G_3$ disproves the following conjecture of Ne\v{s}et\v{r}il and Ossona de Mendez \cite{no}, since long paths have large tree-depth.

\begin{conjecture}[{\cite[Conjecture 7.1]{no}}]\label{con:no}
There exists a constant $t$ such that one can color the vertices of every planar graph by 3 colors in such a way that no monochromatic component will have tree-depth greater than $t$.
\end{conjecture}

We also remark that Theorem \ref{3color minor free} cannot be generalized to graphs with no $H$-topological minor in general.
The following is proved by using an idea of Alon et al.\ \cite{adov}.

\begin{theorem} \label{large for topo}
For positive integers $k,N$, there exists a $(4k-2)$-regular graph $G$ such that for every partition of $V(G)$ into $k$ sets $X_1,X_2,...,X_k$, there exists $i$ with $1 \leq i \leq k$ such that some component of $G[X_i]$ contains at least $N$ vertices.
\end{theorem}

\begin{proof}
It was proved by Erd\H{o}s and Sachs \cite{es} that there exists a $2k$-regular graph $R$ with girth at least $N$.
Since $R$ contains $k\lvert V(R) \rvert$ edges, for any partition of $E(R)$ into $k$ sets, some set contains at least $\lvert V(R) \rvert$ edges and hence induces a subgraph $W$ of $R$ having a cycle.
Since the girth of $R$ is at least $N$, some component of $W$ contains at least $N$ edges.
Therefore, for every partition of $E(R)$ into $k$ sets, there exists a set in the partition such that some component induced by this set contains at least $N$ edges.

Define $G$ to be the line graph of $R$.
So $G$ is $(4k-2)$-regular.
Furthermore, every partition of $V(G)$ into $k$ sets $X_1,X_2,...,X_k$ corresponds to a partition of $E(R)$ into $k$ sets, so there exists $i$ with $1 \leq i \leq k$ such that $G[X_i]$ has a component with at least $N$ vertices.
\end{proof}

Since every graph of maximum degree at most $4k-2$ does not contain any graph with maximum degree at least $4k-1$ as a topological minor, Theorem \ref{large for topo} shows that Theorem \ref{3color minor free} cannot be generalized to topological minor-free graphs in general.

\paragraph{Acknowledgement.} 
The authors would like to thank anonymous referees for their valuable
suggestions, 
and thank Louis Esperet and Gwena\"{e}l Joret 
for bringing Theorem~\ref{no bdd diameter} and Conjecture~\ref{con:no}
to our attention and allowing us to include their proof into this paper. 
The authors would also like to thank Dong Yeap Kang for reading this manuscript carefully and finding a few typos.

\end{document}